\documentclass{article}

\usepackage{graphicx}
\usepackage{amsmath,amsthm,amssymb}
\usepackage{float}
\usepackage{subfig}
\usepackage{url}
\usepackage{bm}

\title{Cofinite Induced Subgraphs of Impartial Combinatorial Games: An Analysis of CIS-Nim}
\author{Scott M. Garrabrant\footnote{Pitzer College, Claremont, CA 91711;  Current address:  Department of Mathematics, UCLA, Los Angeles, CA 90095}, Eric J. Friedman\footnote{International Computer Science Institute, Berkeley CA 94720, USA}, and Adam Scott Landsberg\footnote{W.M. Keck Science Department, Claremont McKenna, Pitzer, and Scripps Colleges, Claremont, CA 91711}}

\begin{document}
\maketitle
\begin{abstract}
Given an impartial combinatorial game $G$, we create a class of related games (CIS-$G$) by specifying a finite set of positions in $G$ and forbidding players from moving to those positions  (leaving all other game rules unchanged).  Such  modifications amount to taking {\em cofinite induced subgraphs} (CIS) of the original game graph. Some recent numerical/heuristic work has suggested that the underlying structure and behavior of such ``CIS-games'' can shed new light on, and bears interesting relationships with, the original games from which they are derived.   In this paper we present an analytical treatment of the cofinite induced subgraphs associated with the game of (three-heap) Nim. This constitutes one of the simplest nontrivial cases of a CIS game.  Our main finding is that although the structure of the winning strategies in games of CIS-Nim can differ greatly from that of Nim, CIS-Nim games inherit a type of period-two scale invariance from the original game of Nim.
\end{abstract}

\newtheorem{thm}{Theorem}[section]
\newtheorem{thm2}{Theorem}[section]
\newtheorem{cor}[thm]{Corollary}
\newtheorem{lem}[thm]{Lemma}
\newtheorem{con}[thm]{Conjecture}
\newtheorem{clm}[thm]{Claim}
\theoremstyle{remark}
\newtheorem{rem}[thm]{Remark}

\theoremstyle{definition}
\newtheorem{defn}[thm]{Definition}
\section{Introduction}

Questions surrounding the underlying structure of the $N$- and $P$-positions in impartial combinatorial games (and associated issues of complexity and optimal strategies) continue to pose substantive challenges to researchers in the field.  For some impartial games, the $N$- and $P$-positions form readily characterizable patterns (such as in Nim, as shown by Bouton's analysis \cite{4}), while for others the structure is much more complex and appears to be resilient against standard analytical treatments (such as the game of Chomp). Indeed, a great deal of work has been devoted to understanding and characterizing the $N$- and $P$- positions in a variety of different games (see, e.g., \cite{3,15,16, 17, 18}).  However, rather than considering an individual impartial game in isolation, recent work \cite{7,8,12} suggests that new and sometimes surprising insights can be had by considering a given game {\em within the context of a family of `closely related' games}.   In particular, for a given  impartial game, the idea is to construct a set of similar games whose game graphs are all `close' to that of the original game (in some suitably defined metric).  One then examines how  the underlying structure of the $N$- and $P$- positions in this associated family of games compares to that of the original.  Indeed, this is the premise behind the earlier notion of ``generic games'' first introduced in \cite{8}:  Given an impartial combinatorial game $G$, one can create slightly perturbed versions of the original game by selecting a finite number of $P$-positions in $G$ and declaring them to be automatic $N$ positions. The class of games formed by arbitrary perturbations of this type has been dubbed the ``generic'' form of the game $G$.  The generic forms of Chomp, Nim, and Wythoff's games have been previously investigated using a combination of numerical methods and (nonrigorous) renormalization techniques from physics \cite{7,8}. It has been observed that in some cases (e.g., three-row Chomp) the original game and its associated family of generic games all share a similar underlying structure, which in turn has yielded a novel geometric characterization of Chomp's $N$- and $P$-positions. In other cases (e.g., three-heap Nim) it has been found that the family of generic games appears to have a rather different underlying structure from the original game. (See also \cite{14} for an alternative discussion of perturbed games.)

The present work on cofinite induced subgraph (CIS) games is a formalization and extension of some of this earlier work on generic games. Using three-heap Nim as a case study, it provides a new approach which not only yields novel results but for the first time allows 
{\em rigorous} mathematical statements to be made about the structure of the N and P positions in this family of Nim-like games.  In particular, Figure 1b illustrates the structure of the $P$-positions in ordinary (three-heap) Nim, while Figure 1a shows an example of a CIS-Nim game (these figures will be discussed more fully later).  Despite the striking structural differences between the two,  we prove that the overall structure of $P$-positions in CIS-Nim exhibits the same  `period-two scale invariance' (to be defined more precisely later) as  Nim. This work constitutes the first formal proofs regarding the properties of CIS-Nim and its relationship to Nim -- relationships that were conjectured to exist based on nonrigorous techniques from physics but never formally proven. Moreover, the proofs themselves, although geared for CIS-Nim, provide more general insights into other impartial games
and suggest a means of determining which structural properties of a game's $P$-positions are unstable and dependent on its specific end-game positions, and which properties are stable and independent of the details of the end game.  
\section{Background}
\subsection{Game Graphs}
Impartial combinatorial games are often represented as directed graphs called ``game graphs'' wherein the vertices of the game graph represent the possible positions of the game and there is a directed edge from vertex $u$ to vertex $v$ if and only if there is a legal move from $u$ to $v$. In this case, we will call $u$ a {\em parent} of $v$ and $v$ a {\em child} of $u$. Starting from any vertex in the game graph, the two players will alternate in moving along any directed edge from the current vertex to another vertex. If the current position has zero out-degree, the player whose turn it is to move has no legal options and is declared the loser. All game graphs of impartial combinatorial games are acyclic and have the property that from any given position there are only finitely many positions which are reachable using any sequence of moves. However, since we will think of these games generally and not limit ourselves to a single starting position, the game graphs we consider will not necessarily be finite. In fact, all of the game graphs discussed in this paper will have infinitely many vertices. 
\subsection{$\bm{P}$- and $\bm{N}$- Positions}
It follows from Zermelo's theorem \cite{11} that from any position either the next player to move can guarantee herself a win under optimal play, or the previous player can guarantee himself a win.  Any position in which the {\em next} player to move can force a win is known as an ``$N$-position,'' while if the {\em previous} player can force the win the position is called a ``$P$-position.''  
This partition of positions into $P$-positions and $N$-positions has the property that no $P$-position has a $P$-position child and every $N$-position has at least one $P$-position child. Further, this is the only partition which satisfies this property. Most importantly, 
knowledge of the $N$- and $P$-positions of a game defines an optimal strategy for the game:  A player at an $N$-position need only move his/her opponent to a $P$-position whenever possible to guarantee a win. 
This means that given a game, a primary goal is to determine the unique partition of positions into $P$-positions and $N$-positions. 
\section{Cofinite Induced Subgraph Games}
Once the the positions of a game graph are partitioned into $P$-positions and $N$-positions, one interesting question is to ask how stable this partition is to minor perturbations made in the game graph. One simple way of making perturbations in a game graph is to remove some finite number of vertices, resulting in an cofinite induced subgraph of the original game graph. 
\begin{defn}\label{defF}
Let $G$ be a game graph, and let $F$ be a finite set of vertices in $G$, called the set of ``forbidden positions''. Let $G-F$ denote the game whose game graph is the induced subgraph formed by removing from $G$ the vertices in $F$ and all edges incident to vertices in $F$. For a given game $G$, ``Cofinite Induced Subgraph $G$'' or ``CIS-$G$'' will refer to the general class of games of the form $G-F$ for any $F$.
\end{defn}
Loosely  speaking, the game $G-F$ is effectively $G$, except that players are forbidden from moving to any position in $F$. 
\begin{rem}\label{misere} If we define $F$ to be the set of all vertices in $G$ which do not have any children, then the game $G-F$ is equivalent to playing $G$ under mis\`{e}re play.
\end{rem}
  
\section{Nim and CIS-Nim}
\subsection{Nim}
Nim \cite{4} is a game played with multiple heaps of beans. Two players alternate taking any positive number of beans from any one heap. When all heaps are empty, the player whose turn it is to play has no move and is therefore declared the loser. 
We will restrict our attention to games of Nim with three heaps and we will consider the three heaps of beans unlabeled so that the positions in this game can be thought of as three-element multisets of non-negative integers, where the three numbers represent the number of beans in the three heaps. The children of a position $\{x,y,z\}$, are all positions of the form $\{x^\prime,y,z\}$ with $x^\prime<x$, $\{x,y^\prime,z\}$ with $y^\prime<y$, or $\{x,y,z^\prime\}$ with $z^\prime<z$. From now on, unless otherwise stated, the word ``Nim'' will refer to three-heap Nim with unlabeled heaps.
\subsection{CIS-Nim}
\begin{figure}
  \centering
    \subfloat[The $P$-positions of Nim-$\{\{1,1,0\}\}$ of the form $\{x,y,z\}$ with $x,y<5000$.]{\includegraphics[width=.7\textwidth]{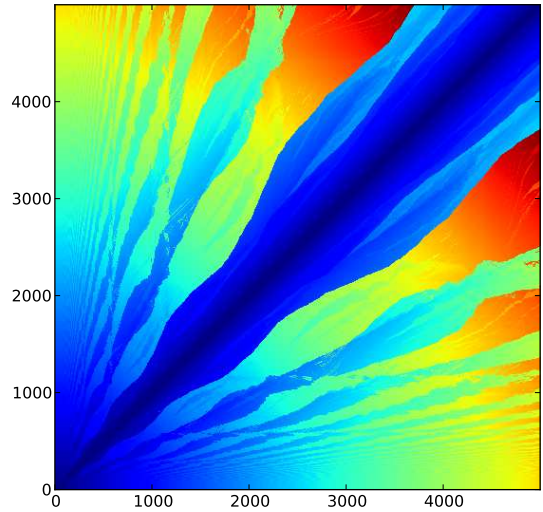}}\\
    \subfloat[The $P$-positions of Nim of the form $\{x,y,z\}$ with $x,y<5000$.]{\includegraphics[width=.35\textwidth]{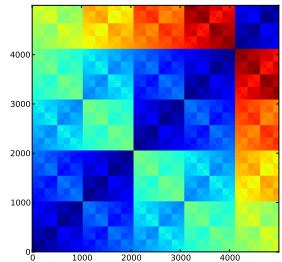}}
\quad
    \subfloat[The $P$-positions of Nim-$\{\{1,1,0\}\}$ of the form $\{x,y,z\}$ with $x,y<2500$.]{\includegraphics[width=.35\textwidth]{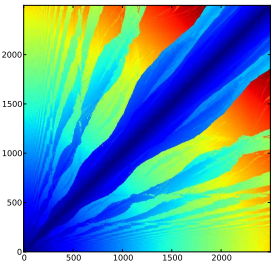}}
\caption{The structure of $P$-positions in Nim and Nim-$\{\{1,1,0\}\}$. Point $(x,y)$ is given with a color representing the unique $z$ such that $\{x,y,z\}$ is a $P$-position. Colors closer to red on the color spectrum represent larger $z$ values, and are normalized based on the largest value of $z$ in each figure.}
\end{figure}
Our goal is to analyze the class of Cofinite Induced Subgraph Nim, or CIS-Nim. Figures 1.a and 1.b show the structure of the $P$-positions in the standard game of Nim and in another instance of Cofinite Induced Subgraph Nim, or CIS-Nim. These structures are remarkably different. The data suggests that the structure of any instance of CIS-Nim looks like one of these two. Most games look similar to Figure 1.a. However, in special cases where none of the forbidden positions are $P$-positions of Nim or the forbidden positions are set up to correct any errors they introduce, the structure will look similar to 1.b. This is because games like Nim are unstable special cases in the generic class of CIS-Nim Games \cite{8,12}.

There is one significant property of Nim which also holds for all instances of Nim: The structure of $P$-positions is invariant up to scaling by a factor of two. Figures 1.a and 1.c demonstrate this period-two scale invariance.

A somewhat weaker but more formal way to state this period-two scale invariance is that given any instance, $\mathrm{Nim}-F$ of CIS-Nim, if we let $\pi(n)$ denote the number of $P$-positions of the form $\{x,y,z\}$, with $x$, $y$, and $z$ all less than $n$, then for any positive integer $n$, $\displaystyle\lim_{k\rightarrow\infty}\frac{\pi(n2^k)}{(n2^k)^2}$ converges to a nonzero constant. In the case of Nim, this can be shown directly using Bouton's well known analytical solution to Nim \cite{4}. The primary result of this paper is a proof that this period-two scale invariance holds for any game of CIS-Nim.
\section{Basic Properties of CIS-Nim}
Before we can prove the period-two scale invariance, we will have to establish some basic properties of CIS-Nim games. 
In Figure 1, the point $(x,y)$ is given a color representing $z$, where $z$ is the unique $P$-position of the form $\{x,y,z\}.$ These figures are only well defined because such a unique $P$-position of the form $\{x,y,z\}$ is known to always exist.
\begin{thm}\label{rowcol} 
Given any instance $\mathit{Nim}-F$ of CIS-Nim, for any nonnegative integers $x$ and $y$ there is a unique $z$ such that $\{x,y,z\}$ is a $P$-position in $\mathit{Nim}-F$. This value of $z$ satisfies the inequality $z\leq x+y+|F|$.
\end{thm}
\begin{proof}
To show uniqueness, assume by way of contradiction that there existed two $P$-positions $\{x,y,z_1\}$ and $\{x,y,z_2\}$. Without loss of generality, assume that $z_1<z_2$. This means that $\{x,y,z_1\}$ is a $P$-position child of $P$-position $\{x,y,z_2\}$, contradicting the fact that no $P$-position has a $P$-position child.

Next, assume by way of contradiction that every position of the form $\{x,y,z\}$ with $z\leq x+y+|F|$ is an $N$-position. There are $x+y+|F|+1$ values of $z$ satisfying this condition, and for all but at most $|F|$ of them, $\{x,y,z\}$ is valid position in $Nim-F$. There are at least $x+y+1$ $N$-positions of this form and each of these positions, $\{x,y,z\}$, therefore has a $P$-position child. This child cannot be of the form $\{x,y,z^\prime\}$ with $z^\prime<z$, so it must be of the form $\{x^\prime,y,z\}$ with $x^\prime<x$ or $\{x,y^\prime,z\}$ with $y^\prime<y$. There are $x$ different pairs of the form $(x^\prime,y)$ with $x^\prime<x$ and we know that for each of these pairs, there is at most one value of $z$ such that $\{x^\prime,y,z\}$ is a $P$-position. Similarly, there are $y$ different pairs of the form $(x,y^\prime)$ with $y^\prime<y$ and we know that for each of these pairs, there is at most one value of $z$ such that $\{x,y^\prime,z\}$ is a $P$-position. There are $x+y+1$ $N$-positions of the form $\{x,y,z\}$. At most $x$ of these positions can have a $P$-position child of the form $\{x^\prime,y,z\}$, and at most $y$ of them can have a $P$-position child of the form $\{x,y^\prime,z\}$. Therefore, at least one of them  has no $P$-position child, contradicting the fact that every $N$-position has a $P$-position child. Therefore, there is at least one $P$-position of the form $\{x,y,z\}$ with $z\leq x+y+|F|$. This means that there is exactly one $P$-position of the form $\{x,y,z\}$, and it satisfies the inequality $z\leq x+y+|F|$.
\end{proof}
The bound of $z\leq x+y+|F|$ given in Theorem \ref{rowcol} is only necessary for small valued $P$-positions. For all but finitely many $P$-positions near $\{0,0,0\},$ we can improve this bound to $z\leq x+y.$
\begin{defn}\label{defC} 
Given an instance $\mathrm{Nim}-F$ of CIS-Nim, let $F_{max}$ equal the largest element (largest number of beans in a single heap) of any position in $F$.
\end{defn}
\begin{thm}\label{rowcol2}
Given any instance $\mathit{Nim}-F$ of CIS-Nim, if $\{x,y,z\}$ is a $P$-position with $z>2F_{max}+|F|$, then $z\leq x+y$
\end{thm}
\begin{proof}
Since $\{x,y,z\}$ is a $P$-position, we know from Theorem \ref{rowcol}, that $z\leq x+y+|F|$. Therefore, $2F_{max}<z-|F|\leq x+y$, so either $F_{max}<x$ or $F_{max}<y$. Therefore, each of the $z$ positions of the form $\{x,y,z^\prime\}$ with $z^\prime<z$ has an element greater than $F_{max}$, and is therefore not in $F$. Since all of these positions are also children of the $P$-position, $\{x,y,z\}$, we know that each of these positions are actually $N$-positions. Therefore, each of these $z$ positions has a distinct $P$-position child of the form $\{x^\prime,y,z^\prime\}$ or $\{x,y^\prime,z^\prime\}$ with $x^\prime<x$, $y^\prime<y$, and $z^\prime<z$. Similarly to in Theorem \ref{rowcol}, there can be at most $x+y$ such $P$-positions, so $z\leq x+y.$
\end{proof}
\begin{cor}\label{xi0}
Given any instance $\mathit{Nim}-F$ of CIS-Nim, for all $n>2F_{max}+|F|$, $\{n,n,0\}$ is a $P$-position.
\end{cor}
\begin{proof}
If $\{n,n,0\}$ were not a $P$-position, it would have a $P$-position child of the form $\{n,n^\prime,0\}$ with $n^\prime<n$. In this case, $n>2F_{max}+|F|$, but $n>n^\prime+0$, contradicting Theorem \ref{rowcol2}.
\end{proof}
In the game of Nim, $\{n,n,0\}$ is a $P$-position for all $n$. This means that if we only consider positions of the form $\{x,y,0\},$ the structures of $P$-positions in Nim and in generic games of CIS-Nim agree on all but finitely many small valued positions. Positions of the form $\{x,y,0\},$ are effectively positions in two-heap Nim games, so Corollary \ref{xi0} tells that two-heap Nim is stable in that large valued $P$-positions are unaffected by removal of small valued positions.

It turns out that if we fix the size of any one heap, the structure of $P$-positions is eventually additively periodic. This is an generalization of Corollary \ref{xi0} which shows that if we fix one heap to be of size 0, the structure of $P$-positions is additively periodic with period 1.  
\begin{clm}\label{periodic}
For any $x$, there exists a $p$ and a $q$ such that for any $y>q$, $\{x,y,z\}$ is a $P$-position if and only if $\{x,y+p,z+p\}$ is a $P$-position.
\end{clm}
We will not prove this claim, as it is technical and unnecessary for our main result. However, the proof is almost identical to an argument given L. Abrams and D. S. Cowen-Morton for a game with similar structure \cite{2}.
\section{Period-Two Scale Invariance in CIS-Nim}
The primary result of this paper is the proof of the following theorem, which is a formulation of the observation that the overall structure of $P$-positions in any game of CIS-Nim is invariant under scaling by a factor of two.  
\begin{thm}[Period-Two Scale Invariance]\label{final}
Given any instance $\mathit{Nim}-F$ of CIS-Nim, let $\pi(n)$ denote the number of $P$-positions (assuming unlabeled heaps) in $\mathit{Nim}-F$ of the form $\{x,y,z\}$, with $x$, $y$, and $z$ all less than $n$. For any positive integer $n$, $\displaystyle\lim_{k\rightarrow\infty}\frac{\pi(n2^k)}{(n2^k)^2}$ converges to a nonzero constant.
\end{thm}
We note that an analog of this result holds for the much simpler case of ordinary Nim (see Figure 2). In the case of ordinary Nim, it is possible to give an explicit formula for $\pi(x)$ as $\frac{3x^2-6xy+4y^2+3x+2}{6}$, where $y$ is the greatest power of 2 less than or equal to $x$. 
\begin{figure}
  \centering
\includegraphics[width=.7\textwidth]{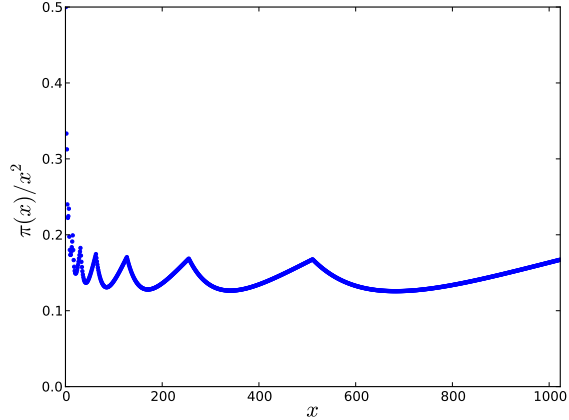}\\
\caption{Plot of $\pi(x)/x^2$ vs. $x$ for ordinary Nim, illustrating a period-two scale invariance. }
\end{figure}

To prove the main theorem, we will first need to prove several Lemmas. 
\subsection{The Set $\bm{S}$}
To start, we will need to think about this problem in terms of a new set $S,$ which encodes much of the information about the structure of the $P$-positions as a set of ordered pairs.
\begin{defn}\label{defA} 
Given any instance $\mathrm{Nim}-F$ of CIS-Nim, let $S$ be the infinite set of ordered pairs of integers such that $(x,y)\in S$ if and only if there exists a $z$ such that $z<y<x$ and $\{x,y,z\}$ is a $P$-position in $\mathrm{Nim}-F$. 
\end{defn}
\begin{defn}\label{defBL} 
Given any instance $\mathrm{Nim}-F$ of CIS-Nim, for any nonnegative integers $x$ and $y$, let $r(x,y)$ be the number of elements of $S$ of the form $(x^\prime,y)$ with $x^\prime\geq x$. Let $b(x,y)$ be the number of elements of $S$ of the form $(x,y^\prime)$ with $y^\prime\leq y$.
\end{defn}
It will be helpful to visualize $S$ as a subset of the plane, as shown in Figure 3. With this visualization in mind, the definitions of $r(x,y)$ and $b(x,y)$ are very natural as the number of points directly to the right or equal to $(x,y)$ and the number of points to below or equal to $(x,y)$ respectively. Notice that there are points $(x,y)\notin S$ with $x>y$ for which $b(x,y)$ is positive. We will refer to such points as ``holes.''
The following two Lemmas will prove some properties of $r(x,y)$ and $b(x,y)$. Lemma \ref{B+2L} captures the essence of why CIS-Nim follows the period-two scale invariance. In fact, if not for the existence of holes, the period-two scale invariance would follow almost directly from Lemma \ref{B+2L}. However, holes do exist, which is why we will need Lemma \ref{B<L} which places limitations on the ways which holes can show up.
\begin{figure}
  \centering
\includegraphics[width=.7\textwidth]{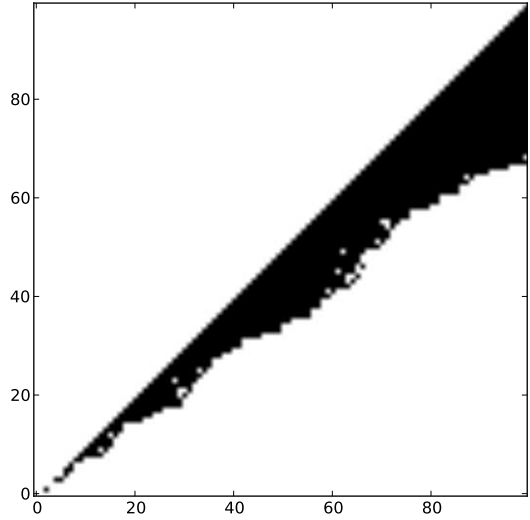}\\
\caption{The set of all points $(x,y)\in S$ with $x,y<100$ for the game Nim-$\{\{1,1,0\}\}$}
\end{figure}
\begin{lem}\label{B+2L}
Given any instance $\mathit{Nim}-F$ of CIS-Nim, for all $x>4F_{max}+3|F|$, $r(x,x)+2b(x,x)+1=x$.
\end{lem}
\begin{proof}
For each of the $x$ values of $y$ satisfying $y<x$, there exists a unique $z$, such that $\{x,y,z\}$ is a $P$-position. For each value of $y$, this unique $z$ will satisfy exactly one of the following: $z>x$, $z=x$, $x>z>y$, $z=y$, or $z<y$.
The number of values of $y$ which satisfy $z>x$ is exactly the number of $P$-positions of the form $\{z,x,y\}$ with $z>x>y$, which is $r(x,x)$. The number of values of $y$ which satisfy $z=x$ is $1$, since $\{x,x,0\}$ is the only $P$-position of the form $\{x,z,y\}$ with $x=z$. The number of values of $y$ which satisfy $x>z>y$ is exactly the number of $P$-positions of the form $\{x,z,y\}$ with $x>z>y$, which is $b(x,x)$. If there were a $P$-position of the form $\{y,z,x\}$ with $y=z$, then $\{y,y,0\}$ would not be a $P$-position, which implies by Corollary \ref{xi0} that $y\leq 2F_{max}+|F|$. By Theorem \ref{rowcol}, this would imply that $x\leq (2F_{max}+|F|)+(2F_{max}+|F|)+|F|=4F_{max}+3|F|$, a contradiction. Therefore, the number of values of $y$ which satisfy $z=y$ is $0$. The number of values of $y$ which satisfy $z<y$ is exactly the number of $P$-positions of the form $\{x,y,z\}$ with $x>y>z$, which is $b(x,x)$. Adding all these together, we get $r(x,x)+2b(x,x)+1$, and we know that the total number of values of $y$ less than $x$ is exactly $x$, so $r(x,x)+2b(x,x)+1=x$.
\end{proof}
\begin{lem}\label{B<L}
Given any instance $\mathit{Nim}-F$ of CIS-Nim, for all $x>y>4F_{max}+3|F|$, if $(x,y)\notin S$ then $b(x,y)\geq r(x,y)$.
\end{lem}
\begin{proof}
There is a element of $S$ of the form $(x^\prime,y)$ with $x^\prime>x$ for each $P$-position of the form $\{x^\prime,y,z\}$ with $z<y<x<x^\prime$. No two of these positions can have the same value for $z$, since then they would have two of the three elements in common, so there would be a move from one to the other. There are therefore $r(x,y)$ distinct values of $z$ for which there is a $P$-position of the form $\{x^\prime,y,z\}$ with $z<y<x<x^\prime$. For each of these values of $z$, $\{x,y,z\}$ cannot be a $P$-position, since it has a $P$-position parent. It therefore must have a $P$-position child. This child cannot be of the form $\{x^{\prime\prime},y,z\}$ with $x^{\prime\prime}<x$ because then it would also be a child of $\{x^\prime,y,z\}$. This child cannot be of the form $\{x,y,z^\prime\}$ since then it would satisfy $x>y>z^\prime$, contradicting the fact that $(x,y)\notin S$. 
Therefore, for each of the $r(x,y)$ values of $z$, there is a $P$-position of the form $\{x,y^\prime,z\}$ with $y^\prime<y$. If two of these $P$-positions, $\{x,y^\prime_1,z_1\}$ and $\{x,y^\prime_2,z_2\}$ were the same, then $y^\prime_1=z_2$ and $y^\prime_2=z_1$. Because $z$ is one of our $r(x,y)$ values, we know that there exist $x_1$ and $x_2$ greater than $x$ such that $\{x_1,y,z_1\}$ and $\{x_2,y,z_2\}$ are both $P$-positions. This would mean that $\{y,z_1,z_2\}$ would have three $P$-position parents, $\{x_1,y,z_1\}$, $\{x_2,y,z_2\}$, and $\{x,z_1,z_2\}$. In order to form a $P$-position by changing any of the three elements of $\{y,z_1,z_2\}$, we would have to increase that element. Therefore, we cannot decrease one element of $\{y,z_1,z_2\}$ to form a $P$-position, so $\{y,z_1,z_2\}$ has no $P$-position children, making it a $P$-position, which contradicts the fact that $\{x,z_1,z_2\}$ is also a $P$-position. Therefore, each of the $r(x,y)$ $P$-positions of the form $\{x,y^\prime,z\}$ with $y^\prime<y$ are unique, and each one has one coordinate equal to $x$, and the other two less than $y$. For each of these $P$-positions, $y^\prime$ and $z$ must be distinct, since otherwise, $4F_{max}+3|F|<x\leq y^\prime+z=2z$, so $z<2F_{max}+|F|$, implying that $\{0,z,z\}$ is a $P$-position child of $\{x,y^\prime,z\}$. Therefore, either $(x,y^\prime)$ or $(x,z)$ is in $S$, so each of the positions will contribute to $S$ an ordered pair of the form $(x,y^{\prime\prime})$ with $y^{\prime\prime}<y$, which will contribute $1$ to $b(x,y)$. Therefore, $b(x,y)\geq r(x,y)$.
\end{proof}
\subsection{The Sets $\bm{U_{x,y}}$, $\bm{{\bar{U}_{x,y}}}$, and $\bm{S_n}$}
In this section, we define a sequence $S_n$ of sets. These sets, and the intermediate sets, $U_{x,y}$ and ${\bar{U}_{x,y}}$ which are used to define $S_n$ encode information about $S.$ In particular, the $S_n$ should be thought of as increasingly accurate approximations of $S$ which are defined to be free of holes. The next three lemmas are building up to proving that for and $n<m,$ it is possible to get from $S_n$ to $S_m$ by changing points in a very limited way. This information about how to construct $S_m$ from $S_n$ will be useful for the next section, where we will prove several important properties of the sequence $S_n.$

The first sets we will need to define on our way to $S_n$ are the $U_{x,y}$. The set $U_{x,y}$ is very similar to the set of all points $(x^\prime,y^\prime)$ in $S$ with $y^\prime<x$. In fact, these two sets have the same size. However, some of the points are moved so that $U_{x,y}$ is free of holes.
\begin{defn}\label{defU}
Given any instance $\mathrm{Nim}-F$ of CIS-Nim, and given any $x>4F_{max}+3|F|$ and $y\leq x$:
\begin{description}
 \item 
Let $A_{x,y}$ be the set of all ordered pairs $(x^\prime,y^\prime)$, such that $0\leq y^\prime<x^\prime\leq x$ and $b(x^\prime,x^\prime)\geq x^\prime-y^\prime$. 
\item
Let $B_{x,y}$ be the set of all ordered pairs $(x,y^\prime)$, such that $0\leq y^\prime<y$ and $b(x,y-1)\geq y-y^\prime$.
\item
Let $C_{x,y}$ be the set of all ordered pairs $(x^\prime,y^\prime)$, such that $y\leq y^\prime <x\leq x^\prime$ and $r(x,y^\prime)> x^\prime-x$.
\item
Let $D_{x,y}$ be the set of all ordered pairs $(x^\prime,y^\prime)$, such that $0\leq y^\prime<y\leq x<x^\prime$ and $r(x+1,y^\prime)> x^\prime-(x+1)$. 
\item
Let $U_{x,y}=A_{x,y}\cup B_{x,y}\cup C_{x,y}\cup D_{x,y}$
\end{description}
\end{defn}
The next lemma will describe the map necessary to get from $U_{x,y}$ to $U_{x,y+1}$. This will be extended in the following two lemmas to describe the map necessary to ger from $S_n$ to $S_m$. 
\begin{lem}\label{x-y}
Given any instance $\mathit{Nim}-F$ of CIS-Nim, and any $x>4F_{max}+3|F|$ and $0\leq y<x$, there exists a bijection from $U_{x,y}$ to $U_{x,y+1}$ which either fixes all elements or fixes all but one element and sends $(x,y-b(x,y))$ to $(x+r(x,y),y)$. This bijection will be the identity if and only if $(x,y)\in S$ or $b(x,y)=0$.
\end{lem}
\begin{proof}
We will partition the set of all points $(x,y)$ with $y^\prime<x^\prime$ into 11 regions. We will show that $U_{x,y}$ and $U_{x,y+1}$ agree for most of these regions. We will see that they do not always agree for regions 3 and 9, but we will show that the way the points in which regions 3 and 9 may differ will exactly follow the statement of the Lemma.
\begin{description}
\item[Region 1 ($\bm{x^\prime< x}$):]
$(x^\prime,y^\prime)\in U_{x,y}$ if and only if $(x^\prime,y^\prime)\in A_{x,y}$ if and only if $b(x^\prime,x^\prime)\geq x^\prime-y^\prime$ if and only if $(x^\prime,y^\prime)\in A_{x,y+1}$ if and only if $(x^\prime,y^\prime)\in U_{x,y+1}$.
\item[Region 2 ($\bm{x^\prime=x}$ and $\bm{y^\prime<y-b(x,y)}$):]
Since $b(x,y-1)\leq b(x,y)< y-y^\prime$, we know that $(x^\prime,y^\prime)\notin B_{x,y}$ so $(x^\prime,y^\prime)\notin U_{x,y}$. Similarly, since $b(x,(y+1)-1)\leq y-y^\prime<(y+1)-y^\prime$, we know that $(x^\prime,y^\prime)\notin B_{x,y+1}$ so $(x^\prime,y^\prime)\notin U_{x,y+1}$.
\item[Region 3 ($\bm{x^\prime=x}$ and $\bm{y^\prime=y-b(x,y)<y}$):]
This case must further be divided into two cases:
\begin{description}
\item[Case 1 ($\bm{(x,y)\in S}$):]
This means that $b(x,y)=b(x,y-1)+1$. Therefore, $(x^\prime,y^\prime)\in U_{x,y}$ if and only if $(x^\prime,y^\prime)\in B_{x,y}$ if and only $b(x,y-1)\geq y-y^\prime$ if and only if $b(x,(y+1)-1)\geq (y+1)-y^\prime$ if and only if $(x^\prime,y^\prime)\in B_{x,y+1}$ if and only if $(x^\prime,y^\prime)\in U_{x,y+1}$

\item[Case 2 ($\bm{(x,y)\notin S}$):]
This means that $b(x,y)=b(x,y-1)$. Therefore, $b(x,y-1)=b(x,y)=y-y^\prime$, so $(x^\prime,y^\prime)\in B_{x,y}$, so $(x^\prime,y^\prime)\in U_{x,y}$.
On the other hand, $b(x,y)=y-y^\prime<y+1-y^\prime$, so $(x^\prime,y^\prime)\notin B_{x,y+1}$ and $(x^\prime,y^\prime)\notin U_{x,y+1}$.
\end{description}
\item[Region 4 ($\bm{x^\prime=x}$ and $\bm{y-b(x,y)<y^\prime<y}$):]
Since $b(x,y)> y-y^\prime$, we know that $b(x,y-1)\geq y-y^\prime$, which implies that $(x^\prime,y^\prime) \in B_{x,y}$ so $(x^\prime,y^\prime)\in U_{x,y}$. Similarly, since $b(x,(y+1)-1)> y-y^\prime$ we know that $b(x,(y+1)-1)\geq (y+1)-y^\prime$ which implies that $(x^\prime,y^\prime)\in B_{x,y+1}$ so $(x^\prime,y^\prime)\in U_{x,y+1}$.
\item[Region 5 ($\bm{x^\prime=x<x+r(x,y)}$ and $\bm{y^\prime=y}$):]
Since $r(x,y^\prime)>x^\prime-x$, we know that $(x^\prime,y^\prime)\in C_{x,y}$, so $(x^\prime,y^\prime)\in U_{x,y}$. 
If $(x,y)\in S$, then $b(x,(y+1)-1)\geq 1=(y+1)-y^\prime$. Otherwise, $(x,y)\notin S$, so by Lemma \ref{B<L}, $b(x,(y+1)-1)\geq r(x,y)\geq 1=(y+1)-y^\prime$. Either way, $(x^\prime,y^\prime)\in B_{x,y+1}$, so $(x^\prime,y^\prime)\in U_{x,y+1}$.
\item[Region 6 ($\bm{x^\prime\geq x}$ and $\bm{y<y^\prime}$):]
$(x^\prime,y^\prime)\in U_{x,y}$ if and only if $(x^\prime,y^\prime)\in C_{x,y}$ if and only if $r(x,y^\prime)> x^\prime-x$ if and only if $(x^\prime,y^\prime)\in C_{x,y+1}$ if and only if $(x^\prime,y^\prime)\in U_{x,y+1}$
\item[Region 7 ($\bm{x^\prime> x}$ and $\bm{y^\prime<y}$):]
$(x^\prime,y^\prime)\in U_{x,y}$ if and only if $(x^\prime,y^\prime)\in D_{x,y}$ if and only if $r(x+1,y^\prime)> x^\prime-(x+1)$ if and only if $(x^\prime,y^\prime)\in D_{x,y+1}$ if and only if $(x^\prime,y^\prime)\in U_{x,y+1}$.
\item[Region 8 ($\bm{x<x^\prime<x+r(x,y)}$ and $\bm{y^\prime=y}$):]
Since $r(x,y^\prime)>x^\prime-x$, we know that $(x^\prime,y^\prime)\in C_{x,y}$, so $(x^\prime,y^\prime)\in U_{x,y}$.
Since $r(x+1,y^\prime)\geq r(x+1,y^\prime)-1>x^\prime-(x+1)$, we know that $(x^\prime,y^\prime)\in D_{x,y+1}$, so $(x^\prime,y^\prime)\in U_{x,y+1}$.
\item[Region 9 ($\bm{x^\prime=x+r(x,y)}$ and $\bm{y^\prime=y}$):]
This region must be further divided into three cases:
\begin{description}
 \item[Case 1 ($\bm{(x,y)\in S}$):]
Since $r(x,y^\prime)=r(x,y)=x^\prime-x$, we know that $(x^\prime,y^\prime)\notin C_{x,y}$, so $(x^\prime,y^\prime)\notin U_{x,y}$.
Since $(x,y)\in S$, $r(x,y)>0$, so we know that $x^\prime>x$. Since $(x,y)\in S$, $r(x+1,y^\prime)=r(x,y)-1=x^\prime-(x+1)$, so $(x^\prime,y^\prime)\notin D_{x,y+1}$, so $(x^\prime,y^\prime)\notin U_{x,y+1}$.
\item[Case 2 ($\bm{b(x,y)=0}$):]
Since $r(x,y^\prime)=r(x,y)=x^\prime-x$, we know that $(x^\prime,y^\prime)\notin C_{x,y}$, so $(x^\prime,y^\prime)\notin U_{x,y}$.
By Lemma \ref{B<L}, since $b(x,y)=0$, we know that $r(x,y)=0$, so $x^\prime=x$. Also, $b(x,(y+1)-1)=0<1=(y+1)-y^\prime$, which implies that $(x^\prime,y^\prime)\notin B_{x,y}$, so $(x^\prime,y^\prime)\notin U_{x,y}$.
\item[Case 3 ($\bm{(x,y)\notin S}$ and $\bm{b(x,y)>0}$):]
Since $r(x,y^\prime)=r(x,y)=x^\prime-x$, we know that $(x^\prime,y^\prime)\notin C_{x,y}$, so $(x^\prime,y^\prime)\notin U_{x,y}$.
If $r(x,y)=0$, then $x^\prime=x$, which since $b(x,(y+1)-1)\geq 1=(y+1)-y^\prime$, which implies that $(x^\prime,y^\prime)\in B_{x,y+1}$. Otherwise, $r(x,y)>0$, and $x^\prime>x$. Since $(x,y)\notin S$, $r(x+1,y^\prime)=r(x,y^\prime)>x+r(x,y^\prime)-(x+1)>x^\prime-(x+1)$, which implies that $(x^\prime,y^\prime)\in D_{x,y+1}$. Either way, $(x^\prime,y^\prime)\in U_{x,y+1}$.
\end{description}
\item[Region 10 ($\bm{x+r(x,y)<x^\prime}$ and $\bm{y^\prime=y}$):]
Since $r(x,y^\prime)\leq x^\prime-x$, we know that $(x^\prime,y^\prime)\notin C_{x,y}$, so $(x^\prime,y^\prime)\notin U_{x,y}$. 
Since $r(x+1,y^\prime)\leq r(x,y^\prime)\leq x^\prime-(x+1)$, we know that $(x^\prime,y^\prime)\notin D_{x,y+1}$, so $(x^\prime,y^\prime)\notin U_{x,y+1}$.
\item[Region 11 ($\bm{y^\prime\geq x}$):]
All of $A_{x,y},B_{x,y},C_{x,y},D_{x,y},A_{x,y+1},B_{x,y+1},C_{x,y+1},$ and $D_{x,y+1}$ are defined to not allow any points in this region, so no points in this region are in either $U_{x,y}$ or $U_{x,y+1}$.
\end{description}
Notice that all points are in $U_{x,y}$ if and only if they are in $U_{x,y+1}$, except for those contained in region 3 case 2, and region 9 case 3. Further, if $b(x,y)>1$ and $b(x,y)\notin S$, then both of these cases contain exactly one point, and otherwise they contain no points. Therefore, if $b(x,y)>0$ or $(x,y)\in S$, then $U_{x,y}=U_{x,y+1}$, and the identity map is a bijection from $U_{x,y}$ to $U_{x,y+1}$. Otherwise, $U_{x,y}$ and $U_{x,y+1}$ are identical, except for the fact that $U_{x,y}$ contains $(x,y-b(x,y))$ while $U_{x,y+1}$ contains $(x+r(x,y),y)$. In this case the map which fixes all but one element and sends $(x,y-b(x,y))$ to $(x+r(x,y),y)$ is a bijection from $U_{x,y}$ to $U_{x,y}$.

Therefore, for any $x>4F_{max}+3|F|$ and $0\leq y<x$, there exists a bijection from $U_{x,y}$ to $U_{x,y+1}$ which either fixes all elements or fixes all but one element and sends $(x,y-b(x,y))$ to $(x+r(x,y),y)$. Further, this bijection will be the identity if and only if $(x,y)\in S$ or $b(x,y)=0$.
\end{proof} 
Next, we will need to introduce the sets ${\bar{U}_{x,y}}$ which are an extention of the sets $U_{x,y}$. In fact, ${\bar{U}_{x,y}}$ is defined from ${\bar{U}_{x,y}}$ by adding infinitely many points so that ${\bar{U}_{x,y}}$ satisfies relations similar to those shown to be satisfied for $S$ in Lenna \ref{B+2L}. Finally, the $S_n$ are just the sets of the form ${\bar{U}_{n,0}}$. 
\begin{defn}\label{Uhat} 
Given any instance $\mathrm{Nim}-F$ of CIS-Nim, and any $x>4F_{max}+3|F|$ and $0\leq y<x$, let ${\bar{U}_{x,y}}$ be the unique set of ordered pairs $(x^\prime,y^\prime)$ with the following two properties:
\begin{description}
 \item[Property 1:] For all $(x^\prime,y^\prime)$ with $y^\prime<x$, $(x^\prime,y^\prime)\in{\bar{U}_{x,y}}$ if and only if $(x^\prime,y^\prime)\in U_{x,y}$.
\item[Property 2:] For all $(x^\prime,y^\prime)$ with $y^\prime\geq x$, $(x^\prime,y^\prime)\in{\bar{U}_{x,y}}$ if and only if $y^\prime<x^\prime\leq 2y^\prime$ and $(y^\prime,\left\lfloor \frac{x^\prime}{2}\right\rfloor)\notin {\bar{U}_{x,y}}$.
\end{description}
Let $S_n$ be a sequence of sets of ordered pairs defined by $S_n={\bar{U}_{n,0}}$.
\end{defn}
\begin{rem}\label{uhatrem}
We know ${\bar{U}_{x,y}}$ exists and is unique since this definition comes with a natural way to determine whether or not $(x^\prime,y^\prime)$ is in ${\bar{U}_{x,y}}$ as a function of the set of all points in ${\bar{U}_{x,y}}$ of the form $(x^{\prime\prime},y^{\prime\prime})$ with $x^{\prime\prime}<x^\prime$.
\end{rem}
We will now extend the result of the previous lemma to these infinite sets, ${\bar{U}_{x,y}}$ and ${\bar{U}_{x,y+1}}$. We find that when a single point is moved in the map from $U_{x,y}$ to $U_{x,y+1}$, this will cause infinitely many points to move in the map from ${\bar{U}_{x,y}}$ to ${\bar{U}_{x,y+1}}$. However, we will show that all of the points moved in this way will move in the same general direction. In particular, if the map sends $(x_1,y_1)$ to $(x_2,y_2)$, then $x_1-y_1\geq x_2-y_2$ and $x_1-2y_1\geq x_2-2y_2$. 
\begin{lem}\label{Uhatxy,xy+1}
Given any instance $\mathit{Nim}-F$ of CIS-Nim, and any $x>4F_{max}+3|F|$ and $0\leq y<x$, there exists a bijection $\phi$ from ${\bar{U}_{x,y}}$ to ${\bar{U}_{x,y+1}}$ such that if $\phi(x_1,y_1)=(x_2,y_2)$, then $x_1-y_1\geq x_2-y_2$ and $x_1-2y_1\geq x_2-2y_2$.  
\end{lem}
\begin{proof}
Let $\phi^\prime$ be the map constructed in Lemma \ref{x-y}. If $\phi^\prime$ fixes all elements, then $U_{x,y}=U_{x,y+1}$, so ${\bar{U}_{x,y}}={\bar{U}_{x,y+1}}$. In this case, the identity map sends ${\bar{U}_{x,y}}$ to ${\bar{U}_{x,y+1}}$ and clearly satisfies the necessary relations.

Otherwise, $\phi^\prime$ sends $(x,y-b(x,y))$ to $(x+r(x,y),y)$, and $(x,y)\notin S$. Consider the map $\phi$ from ${\bar{U}_{x,y}}$ to ${\bar{U}_{x,y+1}}$ defined so that 
$$(2^nx+k_1,2^n(y-b(x,y))+k_2) \mapsto (2^nx+r(x,y)+k_1,2^ny+k_2)$$ and 
$$(2^{n+1}y+k_3,2^n(x+r(x,y))+k_1) \mapsto (2^{n+1}(y-b(x,y))+k_3,2^nx+k_1)$$
 for all $n\geq0$, $0\leq k_1,k_2<2^n$, and $0\leq k_3<2^{n+1}$. Let $\phi$ fix all other elements of ${\bar{U}_{x,y}}$.

First, notice that in this case, $x< 2(y-b(x,y)-1).$ This is because since $(x,y)\notin S$, there are $b(x,y)$ values of $y^\prime<y$ with $(x,y^\prime)\in S$, so there exists at least one $(x,y^\prime)\in S$ with $y^\prime<y-b(x,y)$. There is therefore a $P$-position of the form $\{x,y^\prime,k\}$ with $k<y^\prime<x$. If it were true that $x\geq 2(y-b(x,y)-1),$ then $x\geq 2y^\prime>y^\prime+k$, which since $x\geq 4F_{max}+3|F|,$ contradicts Theorem \ref{rowcol2}. Therefore, $x< 2(y-b(x,y)-1).$

It is easy to verify that this implies that the ordered pairs, $(x^\prime,y^\prime)$ of the form $(2^nx+k_1,2^n(y-b(x,y))+k_2)$, $(2^nx+r(x,y)+k_1,2^ny+k_2)$, $(2^{n+1}y+k_3,2^n(x+r(x,y))+k_1)$, or $(2^{n+1}(y-b(x,y))+k_3,2^nx+k_1)$ all satisfy $x^\prime\leq 2 y^\prime$.

We want to show that the ordered pairs in ${\bar{U}_{x,y}}$ but not in ${\bar{U}_{x,y+1}}$ are exactly those of the form $(2^nx+k_1,2^n(y-b(x,y))+k_2)$ or $(2^{n+1}y+k_3,2^n(x+r(x,y))+k_1)$, and the ordered pairs in ${\bar{U}_{x,y+1}}$ but not in ${\bar{U}_{x,y}}$ are exactly those of the form $(2^nx+r(x,y)+k_1,2^ny+k_2)$ or $(2^{n+1}(y-b(x,y))+k_3,2^nx+k_1)$. We will show that this is true for all ordered pairs $(x^\prime,y^\prime)$ with $y^\prime<m$ by induction on $m$, and it will follow that it holds for all ordered pairs.
\begin{description}
 \item[Base Case: ($\bm{m=x}$):]
In this case, a point $(x^\prime,y^\prime)$ with $y^\prime<m$ is in ${\bar{U}_{x,y}}$ if and only if it is in $U_{x,y}$, and a point $(x^\prime,y^\prime)$ with $y^\prime<m$ is in ${\bar{U}_{x,y+1}}$ if and only if it is in $U_{x,y+1}$. This means that the only point which is in ${\bar{U}_{x,y}}$ but not ${\bar{U}_{x,y+1}}$ is $(x,y-b(x,y))$, which since $2(y-b(x,y))\geq x=m$ is also the only point $(x^\prime,y^\prime)$ of the form $(2^nx+k_1,2^n(y-b(x,y))+k_2)$ or $(2^{n+1}y+k_3,2^n(x+r(x,y))+k_1)$ with $y<x=m$.

Similarly, the only point which is in ${\bar{U}_{x,y+1}}$ but not ${\bar{U}_{x,y}}$ is $(x+r(x,y),y)$, which since $2y\geq 2(y-b(x,y))\geq x=m$ is also the only point $(x^\prime,y^\prime)$ of the form $(2^nx+r(x,y)+k_1,2^ny+k_2)$ or $(2^{n+1}(y-b(x,y))+k_3,2^nx+k_1)$ with $y<x=m$.
\item[Inductive Hypothesis:]
The ordered pairs of the form $(x^\prime,y^\prime)$ with $y^\prime<m$ in ${\bar{U}_{x,y}}$ but not in ${\bar{U}_{x,y+1}}$ are exactly those of the form $(2^nx+k_1,2^n(y-b(x,y))+k_2)$ or $(2^{n+1}y+k_3,2^n(x+r(x,y))+k_1)$, and the ordered pairs $(x^\prime,y^\prime)$ with $y^\prime<m$ in ${\bar{U}_{x,y+1}}$ but not in ${\bar{U}_{x,y}}$ are exactly those of the form $(2^nx+r(x,y)+k_1,2^ny+k_2)$ or $(2^{n+1}(y-b(x,y))+k_3,2^nx+k_1)$.
\item[Inductive Step:]
Consider some arbitrary ordered pair $(x^\prime,y^\prime)$ with $y^\prime=m$. If $x^\prime\geq 2y^\prime$, then $(x^\prime,y^\prime)$ is in neither ${\bar{U}_{x,y}}$ nor ${\bar{U}_{x,y+1}}$, and $(x^\prime,y^\prime)$ is not of the form $(2^nx+k_1,2^n(y-b(x,y))+k_2)$, $(2^{n+1}y+k_3,2^n(x+r(x,y))+k_1)$, $(2^nx+r(x,y)+k_1,2^ny+k_2)$ or $(2^{n+1}(y-b(x,y))+k_3,2^nx+k_1)$.

Otherwise by definition, we know that $(x^\prime,y^\prime)$ is in ${\bar{U}_{x,y}}$ but not in ${\bar{U}_{x,y+1}}$ if and only if $(y^\prime,\left\lfloor\frac{x^\prime}{2}\right\rfloor)$ is in ${\bar{U}_{x,y}}$ but not in ${\bar{U}_{x,y+1}}$, which, since $\left\lfloor\frac{x^\prime}{2}\right\rfloor<y^\prime$, is true if and only if $(y^\prime,\left\lfloor\frac{x^\prime}{2}\right\rfloor)$ is of the form $(2^nx+r(x,y)+k_1,2^ny+k_2)$ or $(2^{n+1}(y-b(x,y))+k_3,2^nx+k_1)$. Notice that $(y^\prime,\left\lfloor\frac{x^\prime}{2}\right\rfloor)$ is of the form $(2^nx+r(x,y)+k_1,2^ny+k_2)$ if and only if $(x^\prime,y^\prime)$ is of the form $(2^{n+1}y+k_3,2^n(x+r(x,y))+k_1)$, and $(y^\prime,\left\lfloor\frac{x^\prime}{2}\right\rfloor)$ is of the form $(2^{n+1}(y-b(x,y))+k_3,2^nx+k_1)$ if and only if $(x^\prime,y^\prime)$ is of the form $(2^nx+k_1,2^n(y-b(x,y))+k_2)$. Therefore, $(x^\prime,y^\prime)$ is in ${\bar{U}_{x,y}}$ but not in ${\bar{U}_{x,y+1}}$ if and only if it is of the form $(2^nx+k_1,2^n(y-b(x,y))+k_2)$ or $(2^{n+1}y+k_3,2^n(x+r(x,y))+k_1)$.

A similar argument shows that $(x^\prime,y^\prime)$ is in ${\bar{U}_{x,y+1}}$ but not in ${\bar{U}_{x,y}}$ if and only if it is of the form $(2^nx+r(x,y)+k_1,2^ny+k_2)$ or $(2^{n+1}(y-b(x,y))+k_3,2^nx+k_1).$
\end{description}
It is easy to verify that ordered pairs of the form $(2^nx+k_1,2^n(y-b(x,y))+k_2)$, $(2^nx+r(x,y)+k_1,2^ny+k_2)$, $(2^{n+1}y+k_3,2^n(x+r(x,y))+k_1)$, or $(2^{n+1}(y-b(x,y))+k_3,2^nx+k_1)$ are all distinct, which is the last thing we need to see that $\phi$ sends the points in ${\bar{U}_{x,y}}$ but not in ${\bar{U}_{x,y+1}}$ bijectively to the points in ${\bar{U}_{x,y+1}}$ but not in ${\bar{U}_{x,y}}$, and fixes all other points in ${\bar{U}_{x,y}}$, so $\phi$ is a bijection from ${\bar{U}_{x,y}}$ to ${\bar{U}_{x,y+1}}$.

Further, 
$$2^nx+k_1-2^n(y-b(x,y))+k_2\geq 2^nx+r(x,y)+2^ny+k_2$$ and 
$$2^{n+1}y+k_3-(2^n(x+r(x,y))+k_1)\geq 2^{n+1}(y-b(x,y))+k_3-(2^nx+k_1),$$
so if $\phi(x_1,y_1)=(x_2,y_2)$, then $x_1-y_1\geq x_2-y_2$. Similarly, 
$$2^nx+k_1-2(2^n(y-b(x,y))+k_2)\geq 2^nx+r(x,y)+k_1-2(2^ny+k_2)$$ and 
$$2^{n+1}y+k_3-2(2^n(x+r(x,y))+k_1)\geq 2^{n+1}(y-b(x,y))+k_3-2(2^nx+k_1),$$
so if $\phi(x_1,y_1)=(x_2,y_2)$, then $x_1-2y_1\geq x_2-2y_2$.

Therefore, $\phi$ is a bijection from ${\bar{U}_{x,y}}$ to ${\bar{U}_{x,y+1}}$ such that if $\phi(x_1,y_1)=(x_2,y_2)$, then $x_1-y_1\geq x_2-y_2$ and $x_1-2y_1\geq x_2-2y_2$.  
\end{proof}
Finally, we will now show that the same properties relating ${\bar{U}_{x,y}}$ and ${\bar{U}_{x,y+1}}$ will also relate $S_n$ and $S_m.$ We will do this by showing that ${\bar{U}_{x,x}}$ is the same as ${\bar{U}_{x+1,0}}$. This will allow us to nest the maps described in Lemma $\ref{Uhatxy,xy+1}$ to get from ${\bar{U}_{x,0}}$ to ${\bar{U}_{x,x}}={\bar{U}_{x+1,0}}$. We will then be able to nest these maps to get from $S_n={\bar{U}_{n,0}}$ to $S_m={\bar{U}_{m,0}}$.
\begin{lem}\label{SnSm}
Given any instance $\mathit{Nim}-F$ of CIS-Nim, and any $m>n>4F_{max}+3|F|$, there exists a bijection $\phi_{n,m}$ from $S_n$ to $S_m$ such that if $\phi_{n,m}(x_1,y_1)=(x_2,y_2)$, then $x_1-y_1\geq x_2-y_2$ and $x_1-2y_1\geq x_2-2y_2$.
\end{lem}
\begin{proof}
It follows directly from the definitions that $A_{x,x}\cup B_{x,x}=A_{x+1,0}$, and that $B_{x+1,0}$, $C_{x,x}$, and $D_{x+1,0}$ are all empty. Also, the set of all points $(x^\prime,y^\prime)$ in $C_{x+1,0}$ with $y^\prime<x$ is exactly $D_{x,x}$. Therefore, $U_{x,x}$ is the set of all points $(x^\prime,y^\prime)$ in $U_{x+1,0}$ with $y^\prime<x$.

Given a point $(x^\prime,y^\prime)$ with $y^\prime=x$ and $y^\prime<x^\prime\leq 2y^\prime$, this point is in ${\bar{U}_{x+1,0}}$ if and only if it is in $U_{x+1,0}$ if and only if it is in $C_{x+1,0}$ if and only if $r(x,y^\prime)> x^\prime-x$.

On the other hand, this point is in ${\bar{U}_{x,x}}$ if and only if $(y^\prime,\left\lfloor \frac{x^\prime}{2}\right\rfloor)\notin U_{x,x}$ if and only if $(y^\prime,\left\lfloor \frac{x^\prime}{2}\right\rfloor)\notin B_{x,x}$ if and only if $b(x,x-1)< x-\left\lfloor\frac{x^\prime}{2}\right\rfloor$.

From Lemma \ref{B+2L}, we know that $r(x,y^\prime)> x^\prime-x$ if and only if  $x-1-2b(x,x)> x^\prime-x$ if and only if $2b(x,x-1)\leq 2x-x^\prime-2$ if and only if $b(x,x-1)\leq x-1-\left\lfloor\frac{x^\prime}{2}\right\rfloor$ if and only if $b(x,x-1)< x-\left\lfloor\frac{x^\prime}{2}\right\rfloor$. Therefore, the point is in ${\bar{U}_{x,x}}$ if and only if it is in ${\bar{U}_{x+1,0}}$.

Finally, notice that points with $y^\prime\geq x^\prime$ are clearly in neither set, and points with $x^\prime>2y^\prime$ are not in ${\bar{U}_{x,x}}$ by definition, and not in ${\bar{U}_{x+1,0}}$, since otherwise $r(x,x)$ would be greater than $x$, making it impossible for $r(x,x)+2b(x,x)+1$ to equal $x$. Therefore, a point with $y^\prime\leq x$ is in ${\bar{U}_{x+1,0}}$ if and only if it is in ${\bar{U}_{x,x}}$. Therefore, ${\bar{U}_{x,x}}$ satisfies property 1 for ${\bar{U}_{x+1,0}}$, and clearly property 2 for ${\bar{U}_{x,x}}$ is stronger than property 2 for ${\bar{U}_{x+1,0}}$. Therefore, ${\bar{U}_{x,x}}$ satisfies both property 1 and property 2 for ${\bar{U}_{x+1,0}}$, so ${\bar{U}_{x,x}}={\bar{U}_{x+1,0}}$.

We know from Lemma \ref{x-y} that for all $y<x$ there is a bijection from ${\bar{U}_{x,y}}$ to ${\bar{U}_{x,y+1}}$ with all the properties described in Lemma \ref{x-y}. There is therefore a bijection from ${\bar{U}_{x,0}}$ to ${\bar{U}_{x,x}}$ which can be expressed as a product of the bijections described in Lemma \ref{x-y}. Since 
${\bar{U}_{x,x}}={\bar{U}_{x+1,0}}$, this means there is a bijection from $S_x={\bar{U}_{x,0}}$ to $S_{x+1}={\bar{U}_{x+1,0}}$ which can be expressed as a composition of the bijections described in Lemma \ref{Uhatxy,xy+1}. By composing these bijections, we get that for any $n<m$ there is a bijection from $S_n$ to $S_m$ which can be expressed as a composition of the bijections described in Lemma \ref{x-y}.

Clearly, any composition of functions described in Lemma \ref{x-y} also satisfy the same relation. We therefore constructed a bijection $\phi_{n,m}$ such that if $\phi_{n,m}(x_1,y_1)=(x_2,y_2)$, then $x_1-y_1\geq x_2-y_2$ and $x_1-2y_1\geq x_2-2y_2$. 
\end{proof}
\subsection{Properties of $\bm{S_n}$}
In this section we will define two functions, $g(n,m)$ and $h(n,m)$ which will contain information about $S_n$. We will use $g$ as a potential function that will limit how much $h(n,n)$ will be able to change as $n$ increases. This will ultimately allow us to prove that $\displaystyle\lim_{k\rightarrow\infty}\frac{
h({n2^k},n2^k)}{4^k}$ converges. This will be helpful, because $h(n,n)$ will allow us to approximately construct $\pi(n)$ and show that $\displaystyle\lim_{k\rightarrow\infty}\frac{\pi(n2^k)}{(n2^k)^2}$ also converges, and prove that CIS-Nim exhibits a period-two scale invariance.
\begin{defn}\label{defh} Given any instance $\mathrm{Nim}-F$ of CIS-Nim, and positive integers $m> 4F_{max}+3|F|$ and $n$, let $R_n$ denote the set of all $(x,y)$ with $y<n\leq x\leq 2y$, and let $h(m,n)=|R_n\cap S_m|.$
\end{defn}
\begin{lem}\label{hprop}
Given any instance $\mathit{Nim}-F$ of CIS-Nim, for all $n>4F_{max}+3|F|$ and for all non-negative integers $k$, $h(n,n2^k)=4^kh(n,n)$
\end{lem}
\begin{proof}
We will proceed by induction on $k$.
\begin{description}
 \item [Base Case ($\bm{k=0}$):]
$h(n,n2^k)=h(n,n)=4^kh(n,n)$
\item[Inductive Hypothesis:]$h(n,n2^{k-1})=4^{k-1}h(n,n)$
\item[Inductive Step:]
Consider the map $\psi:R_{n2^k}\rightarrow R_{n2^{k-1}}$ defined by $(x,y)\mapsto (\left\lfloor\frac{x}{2}\right\rfloor,\left\lfloor\frac{y}{2}\right\rfloor)$. Notice that $y<n2^k\leq x\leq 2y$ if and only if $\left\lfloor\frac{x}{y}\right\rfloor<n2^{k-1}\leq \left\lfloor\frac{x}{2}\right\rfloor\leq 2\left\lfloor\frac{y}{2}\right\rfloor$, so this map is well defined and surjective. Further, if $(x,y)\in R_{n2^k}$, then $y<n2^k\leq x\leq 2y$, so by the definition of $S_n$, we know that 
$$(x,y)\in S_n\Leftrightarrow(y,\left\lfloor\frac{x}{2}\right\rfloor)\notin S_n\Leftrightarrow(\left\lfloor\frac{x}{2}\right\rfloor,\left\lfloor\frac{y}{2}\right\rfloor)\in S_n.$$ Therefore, for any $(x,y)\in R_{n2^k}$ we get that $(x,y)\in S_n$ if and only if $\psi(x,y)\in S_n$. Therefore, $\psi$ maps $R_{n2^k}\cap S_n$ onto $R_{n2^{k-1}}\cap S_n$. Further, there are exactly four points, $(2x,2y)$, $(2x+1,2y)$, $(2x,2y+1)$, and $(2x+1,2y+1),$ which map to the point $(x,y)$. Therefore, $\psi$ maps four points in $R_{n2^k}\cap S_n$ onto each point in $R_{n2^k}\cap S_n$, so $|R_{n2^k}\cap S_n|=4|R_{n2^{k-1}}\cap S_n|$. Therefore, $h(n,n2^{k})=4h(n,n2^{k-1})=4(4^{k-1}h(n,n))=4^{k}h(n,n).$
\end{description}
\end{proof}
We have shown that $h(n,n2^k)=4^kh(n,n)$. We would like to relate $h(n2^k,n2^k)$ to $h(n,n2^k)$, which will allow us to relate $h(n2^k,n2^k)$ to $h(n,n)$. To do this, we will have to limit how much $h(m,n2^k)$ can change as $m$ changes from $n$ to $n2^k$. We will need $g(m,n)$, which will serve as a potential function limiting how much $h(m,n)$ can change as we repeatedly double $m$.
\begin{defn}\label{defgh} Given any instance $\mathrm{Nim}-F$ of CIS-Nim, and positive integers $m> 4F_{max}+3|F|$ and $n$, let $T_n$ be set of ordered pairs of the form $(x,y)$, with $2y-x\leq n$. Let $f((x,y),n)=n+2x-3y+2$, and let $g(m,n)$ denote the sum over all pairs $(x,y)\in T_n\cap S_m$ of $f((x,y),n)$.
\end{defn}
\begin{defn}\label{wellb}
Given any instance $\mathrm{Nim}-F$ of CIS-Nim, if $r(x,x)<x-1$ for all but finitely many $x$, we will say the game is ``well behaved.'' In this case, let $c_1$ be the least natural number, such that $c_1\geq4F_{max}+3|F|$ and $r(x,x)<x-1$ for all $x>c_1$.
\end{defn}
It turns out that all the interesting games of CIS-Nim are well behaved. However, some special cases such as Nim are not well behaved. Games which are not well behaved are much easier to analyze. We will continue our analysis in Lemmas \ref{gprop} and \ref{limith} only considering well behaved games. Then, in Lemma \ref{degenerate}, we will show that the result of \ref{limith} also holds for games which are not well behaved.

We are going to use $g$ as a potential function to limit how much $h$ will be able to change. The following lemma will provide an upper bound for $g({c_1},2^k)$ in terms of $k$, which will give us our initial finite potential.
\begin{lem}\label{gprop}
Given any well behaved instance $\mathit{Nim}-F$ of CIS-Nim, there exists a constant $c_2$ such that for all $k$, $\frac{g({c_1},2^k)}{8^k}\leq c_2$
\end{lem}
\begin{proof}
Consider the map $\psi:T_{2^k}\rightarrow T_{2^{k-1}}$ defined by $(x,y)\mapsto (\left\lfloor\frac{x}{2}\right\rfloor,\left\lfloor\frac{y}{2}\right\rfloor)$. 
Given any point, $(x,y)\in T_{2^k}$, we know that 
$$2\left\lfloor\frac{y}{2}\right\rfloor-\left\lfloor\frac{x}{2}\right\rfloor\leq2\frac{y}{2}+\left\lceil\frac{-x}{2}\right\rceil=\left\lceil\frac{2y-x}{2}\right\rceil\leq\left\lceil \frac{2^k}{2}\right\rceil=2^{k-1}.$$ Therefore, $\psi(x,y)\in T_{2^{k-1}}$.

Similarly, from the definition of $S_{c_1}$, we know that for any point $(x,y)\in S_{c_1}$ with $\left\lfloor\frac{x}{2}\right\rfloor>c_1$, $(y,\left\lfloor\frac{x}{2}\right\rfloor)\notin S_{c_1}$, so $(\left\lfloor\frac{x}{2}\right\rfloor,\left\lfloor\frac{y}{2}\right\rfloor)\in S_{c_1}$. Therefore, $\psi(x,y)\in S_{c_1}$.

For every point $(x,y)\in T_{2^k}$, 
$$f((x,y),2^k)=2^k+2x-3y+2\leq 2(2^{k-1})+4\left\lfloor\frac{x}{2}\right\rfloor-6\left\lfloor\frac{y}{2}\right\rfloor+4=2f(\psi(x,y),2^{k-1}).$$
Combining these three facts, we get that $\psi$ maps each point in $T_{2^k}\cap S_{c_1}$ with $\left\lfloor\frac{x}{2}\right\rfloor>c_1$ to a point in $T_{2^{k-1}}\cap S_{c_1}$. This map is clearly sends at most four points to any given point, and $f((x,y),2^k)\leq 2f(\psi(x,y),2^{k-1})$. Therefore, the sum over all elements $(x,y)$ in $T_{2^k}\cap S_{c_1}$ with $\left\lfloor\frac{x}{2}\right\rfloor>c_1$ of $f((x,y),2^k)$ is at most four times the sum over all elements $(x,y)\in T_{2^{k-1}}\cap S_{c_1}$ of
$2f((x,y),2^{k-1})$. This value equals $8g({c_1},2^{k-1})$.

There are at most $2c_1+2$ values of $x$ with $\left\lfloor\frac{x}{2}\right\rfloor\leq c_1$, and for each of these values, at most $2c_1+2$ values of $y$ with $y<x$. Therefore, there are at most $(2c_1+2)^2$ points $(x,y)\in T_{2^k}\cap S_{c_1}$ with $\left\lfloor\frac{x}{2}\right\rfloor\leq c_1$. For each of these points, $f((x,y),2^k)\leq 2^k+2(2c_1+2)+2$. Therefore, the sum over all elements $(x,y)$ in $T_{2^k}\cap S_{c_1}$ with $\left\lfloor\frac{x}{2}\right\rfloor\leq c_1$ of $f((x,y),2^k)$ is at most $(2^k+2(2c_1+2)+2)(2c_1+2)^2$. Combining this with the last result gives us that the sum over all elements $(x,y)$ in $T_{2^k}\cap S_{c_1}$ of $f((x,y),2^k)$ is at most $8g({c_1},2^{k-1})+(2^k+4c_1+6)(2c_1+2)^2$.

Therefore, $g({c_1},2^{k})\leq 8g({c_1},2^{k-1})+(2^k+4c_1+6)(2c_1+2)^2$, which implies that 
$$\frac{g({c_1},2^{k})}{8^k}\leq \frac{g({c_1},2^{k-1})}{8^{k-1}}+\frac{(2^k+4c_1+6)(2c_1+2)^2}{8^{k-1}}.$$ Therefore, $$\frac{g({c_1},2^{k})}{8^k}\leq g({c_1},1)+\sum_{i=1}^\infty \frac{(2^i+4c_1+6)(2c_1+2)^2}{8^{i-1}}.$$
Because the game is well behaved, $g({c_1},1)$ is a finite constant. We also know $\displaystyle \sum_{i=1}^\infty \frac{(2^i+4c_1+6)(2c_1+2)^2}{8^{i-1}}=\sum_{i=1}^\infty \frac{(8)(2c_1+2)^2}{4^{i}}+\frac{8(4c_1+6)(2c_1+2)^2}{8^{i}}$ is the sum of two geometric series with ratio less than one, and therefore converges to a finite constant. Therefore, $\frac{g({c_1},2^{k})}{8^k}$ is bounded above by some finite constant $c_2$.
\end{proof}
The following lemma is a key part of proving the period-two scale invariance. After this result, all that will remain are a few technical details. We will prove that $\displaystyle\lim_{m\rightarrow\infty}\frac{h({n2^m},n2^m)}{4^m}$ converges, which will later be modified to a similar statement about $\pi$. The general strategy is to assume by way of contradiction that it does not converge, and therefore must contain infinitely many points above and below a interval of positive length. This means that $\frac{h({n2^m},n2^m)}{4^m}$ must increase and decrease by a fixed amount infinitely many times. We will use $g$ as a potential function, and show that as $h$ changes, $g$ must decrease by some fixed amount. This means that $g$ must decrease infinitely, but $g$ will start at a finite potential, and will remain nonnegative, causing a contradiction.
\begin{lem}\label{limith}
Given any well behaved instance $\mathit{Nim}-F$ of CIS-Nim, for any positive integer $n$, the limit  $\displaystyle\lim_{m\rightarrow\infty}\frac{h({n2^m},n2^m)}{4^m}$ converges.
\end{lem}
\begin{proof}
Let $\zeta_m=\frac{h(n2^m,n2^m)}{4^m}$ and assume for the purpose of contradiction that $\zeta_m$ does not converge. For every point $(x,y)\in R_{n2^m}$, we know that $y<n2^m$, and $x<2y<n2^{m+1}$. There are only $2(n2^{m+1})^2$, such points, so $h(n2^m,n2^m)$ is bounded above by $2n^24^{m+1}$, and $\zeta_m$ is bounded above by $8n^2$ Therefore, $\zeta_m$ is a sequence in the compact interval $[0,8n^2]$, which does not converge. There therefore exists some real numbers $q<r$, such that infinitely many terms of $\zeta_m$ are less than $q$ and infinitely many terms of $\zeta_m$ are greater than $r$. There therefore exists some subsequence $\zeta_{s_m}$ such that $\zeta_{s_m}$ is less than $q$ for all odd $m$ and greater than $r$ for all even $m$. Fix some even $m$, and for convenience of notation let $t(m)=n2^{s_m}.$
From Lemma \ref{hprop} we know that for any $k\geq s_{m+1}$, 
$$h({t(m)},n2^k)=h({t(m)},t(m))4^{k-s_m}>r4^k,$$ and $$h({t(m+1)},n2^k)=h({t(m+1)},t(m+1))4^{k-s_{m+1}}<q4^k.$$
We know that $|S_{t(m)}\cap R_{n2^k}|> r4^k$ and $|S_{t(m+1)}\cap R_{n2^k}|<q4^k$ Therefore, there are at least $4^k(r-q)$ points $(x,y)\in S_{t(m)}$, such that $(x,y)\in R_{n2^k}$, but $\phi_{t(m),t(m+1)}(x,y)\notin R_{n2^k}$. 
Given any constant, $d$, at most $dn2^k$ of these $4^k(r-q)$ points can satisfy the relation, $x<n2^k+d$, and at most $dn2^k$ of them can satisfy the relation $y\geq n2^k-d$. Now notice that if $(x_1,y_1)$ is one of the remaining $4^k(r-q)-2dn2^k$ points, and $\phi_{t(m),t(m)}(x_1,y_1)=(x_2,y_2)$, then either $x_2\leq n2^k<x_1-d$ or $y_2>n2^k\geq y_1+d$.

If $x_2<x_1-d$, then since we also know that $x_1-2y_1\geq x_2-2y_2$ algebra shows that $(2x_1-3y_1)-(2x_2-3y_2)\geq \frac{d}{2}$. On the other hand, if $y_2\geq y_1+d$, then since we also know that $x_1-y_1\geq x_2-y_2$ algebra shows that $(2x_1-3y_1)-(2x_2-3y_2)\geq d$. Either way, $(n2^k+2x_1-3y_1)-(n2^k+2x_2-3y_2)\geq \frac{d}{2}$, so $f((x_2,y_2),n2^k)\leq f((x_1,y_1),n2^k)-\frac{d}{2}$. Therefore, for at least $4^k(r-q)-2dn2^k$ points, $(x,y)$, we have the relation $f(\phi_{t(m),t(m+1)}(x,y),n2^k)\leq f((x,y),n2^k)-\frac{d}{2}$. Further, for all of these points, we have know that $y\leq n2^k-d$ and $x\geq n2^k+d$, so $f((x,y),n2^k)\geq n2^k+2(n2^k+d)-3(n2^k-d)=5d\geq \frac{d}{2}$. Each of these $4^k(r-q)-2dn2^k$ contribute to $g({t(m)},n2^k)$, and whether or not they contribute to $g({t(m+1)},n2^k)$, the contribution for each point is reduced by at least $\frac{d}{2}$. It is easy to see that $f(\phi_{t(m),t(m+1)}(x,y),n2^k)\leq f((x,y),n2^k)$ for all $(x,y)$, so every point which contributes to $g({t(m+1)},n2^k)$ will contribute at least as much to $g({t(m)},n2^k)$. It is also easy to see that no point can contribute negatively to $g({t(m)},n2^k)$. Therefore, the contribution of at least $4^k(r-q)-2dn2^k$ decreases by at least $\frac{d}{2}$ and no point increases its contribution, so we know that
$$g({t(m+1)},n2^k)\leq g({t(m)},n2^k)-\frac{d}{2}(4^k(r-q)-2dn2^k).$$
In particular, if we let $d=\frac{2^k(r-q)}{4n}$, we get that
$$g({t(m+1)},n2^k)\leq g({t(m)},n2^k)-\frac{8^k(r-q)^2}{16n}.$$
It is easy to see that $f(\phi_{t(m+1),t(m+2)}(x,y),n2^k)\leq f((x,y),n2^k)$ for all $(x,y)$, so every point which contributes to $g({t(m+2)},n2^k)$ will contribute at least as much to $g({t(m+1)},n2^k)$. It is also easy to see that no point can contribute negatively to $g({t(m+1)},n2^k)$. Therefore, $g({t(m+2)},n2^k)\leq g({t(m+1)},n2^k)$, and since $m$ was an arbitrary even integer, we can nest this relation multiple times, to get that for any $m$ and any $k>s_m$
$$g({t(m)},n2^k)\leq g({t(0)},n2^k)-m\frac{8^k(r-q)^2}{16n}.$$
We know that $g({t(m)},n2^k)\geq 0$, and that 
$$g({t(0)},n2^k)\leq g(c_1,n2^k)\leq g(c_1,2^{k+\left\lceil\log_2(n)\right\rceil})\leq g({c_1},1) 8^{k+\left\lceil\log_2(n)\right\rceil}.$$
Combining these three relations, we get that 
$$m\frac{8^k(r-q)^2}{16n}\leq 8^{k+\left\lceil\log_2(n)\right\rceil}c_2\leq  8^{k+\log_2(n)+1}c_2=8^{k+1}n^3c_2.$$
Now, finally, if we choose $m$ such that $m>\frac{128n^4c_2}{(r-q)^2}$, we get that 
$$8^{k+1}n^3c_2=\frac{128n^4c_2}{(r-q)^2}\frac{8^k(r-q)^2}{16n}<m\frac{8^k(r-q)^2}{16n}\leq 8^{k+1}n^3c_2.$$
This is a contradiction, implying that $\zeta_m$ does converge.
\end{proof}
Now, we will show that the result we just we can reach the same conclusion we just reached for well behaved games in games which are not well behaved.
\begin{lem}\label{degenerate}
Given any instance $\mathit{Nim}-F$ of CIS-Nim which is not well behaved, for any positive integer $n$, the limit $\displaystyle\lim_{k\rightarrow\infty}\frac{h({n2^k},n2^k)}{4^k}$ converges.
\end{lem}
\begin{proof}
The methods here will be very different. Games which are not well behaved are much easier to analyze, so we will be able to describe the $S_n$ in great detail, and the fact that $\displaystyle\lim_{k\rightarrow\infty}\frac{h({n2^k},n2^k)}{4^k}$ converges will follow directly.

We know that for infinitely many values of $m>4F_{max}+3|F|+1$, $r(m,m)\geq m-1$. It is not possible to have $r(m,m)>m-1$, since $r(m,m)+2b(m,m)+1=m$. Therefore, for infinitely many values of $m>4F_{max}+3|F|+1$, $r(m,m)= m-1$. For any such $m$, we know that $r(m,m)+2b(m,m)+1=m$, so $b(m,m)=0$. This means that for all $y< m$, $(m,y)\notin S$ and $b(m,y)=0$. Therefore, $r(m,y)=0$, which means that for any $y< m\leq x$, $(x,y)\notin S$. In particular, this means that $r(m-1,m-1)=r(m,m-1)=0$, so $2b(m-1,m-1)+0+1=m-1$, so $b(m-1,m-1)=\frac{m-2}{2}$.

Notice that if there were some $m^\prime\leq \frac{m-1}{2}$, with $(m-1,m^\prime)\in S$, then there would be a $P$-position of the form $\{m-1,m^\prime,m^{\prime\prime}\}$ with $m^{\prime\prime}<m^\prime$. In this case, $m-1>4F_{max}+3|F|\geq2F_{max}+|F|$ and $$m^\prime+m^{\prime\prime}< 2m^\prime\leq 2(\frac{m-1}{2})=m-1,$$ contradicting Theorem \ref{rowcol2}. Therefore, for all $(m-1,m^\prime)\in S$, $m^\prime> \frac{m-1}{2}$. This means that the $\frac{m-2}{2}$ values of $m^\prime$ with $(m-1,m^\prime)\in S$ are exactly the integers from $\frac{m}{2}$ to $m-2$ inclusive. Therefore, $(m-1,\frac{m}{2})\in S$.

Now, we want to show that $(m^\prime,\frac{m}{2})\in S$ for all $\frac{m}{2}<m^\prime\leq m-1$. Assume for the purpose of contradiction that there exists some $\frac{m}{2}<m^\prime<m-1$, such that $(m^\prime,\frac{m}{2})\notin S$, and consider the greatest such $m^\prime$. $r(m^\prime,\frac{m}{2})=m-1-m^\prime$. Therefore, by Lemma \ref{B<L}, $b(m^\prime,\frac{m}{2})\geq m-1-m^\prime$, so there exists some $m^{\prime\prime}\leq m^\prime-\frac{m}{2}$ with $(m^\prime,m^{\prime\prime})\in S$. There would therefore exist a $P$-position of the form $\{m^\prime,m^{\prime\prime},m^{\prime\prime\prime}\}$ with $m^{\prime\prime}>m^{\prime\prime\prime}$. Notice that $m^\prime>\frac{m}{2}>2F_{max}+|F|$ and $$m^{\prime\prime}+m^{\prime\prime\prime}< 2m^{\prime\prime}\leq2(m^\prime-\frac{m}{2})=2m^\prime-m< m^\prime+(m^\prime-1)-m<m^\prime,$$  contradicting Theorem \ref{rowcol2}. Therefore, $(m^\prime,\frac{m}{2})\in S$ for all $\frac{m}{2}<m^\prime\leq m-1$, so $r(\frac{m}{2},\frac{m}{2})\geq m-1-\frac{m}{2}$, so $r(\frac{m}{2},\frac{m}{2})=\frac{m}{2}-1.$

For any point $m$ satisfying $r(m,m)=m$, we know the greatest value $m^\prime<m$ satisfying $r(m^\prime,m^\prime)=m^\prime-1$ is $\frac{m}{2}$. This also tells us that the least value $m^\prime<m$ satisfying $r(m^\prime,m^\prime)=m^\prime-1$ is $2m$. Therefore, if we let $m$ be the least value greater than $4F_{max}+3|F|+1$ with $r(m,m)=m-1$, then for any $m^\prime>4F_{max}+3|F|+1$, $r(m^\prime,m^\prime)=m^\prime-1$ if and only if $m^\prime = m2^k$ for some nonnegative integer $k$.

We also know that for any $y< m2^k\leq x$, $(x,y)\notin S$, which implies that $r(y,y)\leq m2^k-y-1$, and $b(x,x)\leq x-m2^k$. Therefore, given any $m2^{k-1}\leq m^\prime<m2^k$, $r(m^\prime,m^\prime)\leq m2^k-m^\prime-1$ and $b(m^\prime,m^\prime)\leq m^\prime-m2^{k-1}$. Notice that $r(m^\prime,m^\prime)+2b(m^\prime,m^\prime)+1=m^\prime$ is satisfied if and only if both of these inequalities are tight, so given any $m2^{k-1}\leq m^\prime<m2^k$, $r(m^\prime,m^\prime)=m2^k-m^\prime-1$ and $b(m^\prime,m^\prime)=m^\prime-m2^{k-1}$.

This means that all of the points $(x,y)$ which we have not already determined to not be in $S$ must be in $S$. Therefore, for any $x>4F_{max}+3|F|+1, (x,y)\in S$ if and only if there exists a $k$, such that $m2^{k-1}\leq x,y<m2^k$. This in particular means that there is no $(x,y)$ with $x>4F_{max}+3|F|+1$, such that $(x,y)\notin S$ but $b(x,y)>0$. This means that all of the functions defined in Lemmas \ref{x-y}, \ref{Uhatxy,xy+1}, and \ref{SnSm} are the identity, which in particular means that if we set $k^\prime$ to be the least nonnegative integer such that $n2^{k^\prime}>4F_{max}+3|F|+1$, then $S_{n2^k}=S_{n2^{k^\prime}}$ for all $k\geq k^\prime$. Therefore, $$\lim_{k\rightarrow\infty}\frac{
h({n2^k},n2^k)}{4^k}=\lim_{k\rightarrow\infty}\frac{
h({n2^{k^\prime}},n2^k)}{4^k}=\lim_{k\rightarrow\infty}\frac{
h({n2^{k^\prime}},n2^{k^\prime})}{4^{k^\prime}}=\frac{h({n2^{k^\prime}},n2^{k^\prime})}{4^{k^\prime}}.$$
Therefore, $\displaystyle\lim_{k\rightarrow\infty}\frac{
h({n2^k},n2^k)}{4^k}$ converges.
\end{proof}
\subsection{Proof of the Period-Two Scale Invariance}
We now have all the lemmas necessary to complete the proof of the period-two scale invariance. 
\begin{thm2}[Period-Two Scale Invariance]\label{final2}
Given any instance $\mathit{Nim}-F$ of CIS-Nim, let $\pi(n)$ denote the number of $P$-positions in $\mathit{Nim}-F$ of the form $\{x,y,z\}$, with $x$, $y$, and $z$ all less than $n$. For any positive integer $n$, $\displaystyle\lim_{k\rightarrow\infty}\frac{\pi(n2^k)}{(n2^k)^2}$ converges to a nonzero constant.
\end{thm2}
\begin{proof}
All that needs to be done to complete the proof is to convert the result from Lemmas \ref{limith} and \ref{degenerate} from a statement about $h(m,m)$ to an analogous statement about $\pi(n)$.

Notice that $h({m},m)$, is the number of $(x,y)\in S_{m}$ with $y<m\leq x\leq 2y$. From the definition of $S_{m}$, this is the number of ordered pairs $(x,y)$ with $y<m\leq x\leq 2y$ and $r(m,y)>x-m$.

Notice that if there existed a point $(x,y)$ with $y<m\leq x$ and $x\geq 2y$ such that $r(m,y)>x-m$, then there must be greater than $2y+1-m$ values of $x^\prime\geq m$ with $(x^\prime,y)\in S$, then there must be at least 1 value of $x^\prime\geq 2y$ with $(x^\prime,y)\in S$. However, this would mean that there would be a $P$-position of the form $\{x,y,z\}$ with $x\geq 2y$ and $x>y>z$. However, this means that $x\geq 2y>y+z$, which, since $x\geq m>4F_{max}+3|F|$, contradicts Theorem \ref{rowcol2}. This means that the $x\leq 2y$ condition is unnecessary, so $h(m,m)$ is the number of ordered pairs $(x,y)$ with $y<m\leq x$ and $r(m,y)>x-m$.

Notice that for each $y$, there are exactly $r(m,y)$ values of $x$ with $m\leq x$ and $r(m,y)>x-m$. Therefore, $\displaystyle h(m,m)=\sum_{y=0}^{m-1} r(m,y),$ which is exactly the number of ordered pairs $(x,y)\in S$ with $y<m\leq x$, or equivalently the number of $P$-positions of the form $\{x,y,z\}$ with $x\geq m>y>z$.

For each ordered pair $(y,z)$ with $z<y<m$, let $x$ be the unique value such that $\{x,y,z\}$ is a $P$-position. There are exactly $\frac{m^2-m}{2}$ such ordered pairs, and exactly $h(m,m)$ of them satisfy the relation $x\geq m$. Therefore, the remaining $\frac{m^2-m}{2}-h(m,m)$ of them satisfy the relation $x<m$. 
Let $\pi_3(m)$ be the number of $P$-positions of the form $\{x,y,z\}$ with $x$, $y$, and $z$ distinct and less than $m$. Let Let $\pi_2(m)$ be the number of $P$-positions of the form $\{x,x,y\}$ with $x$ and $y$ distinct and less than $m$. Let $\pi_1(m)$ be the number of $P$-positions of the form $\{x,x,x\}$ with $x<m$. Clearly, $\pi_3(m)+\pi_2(m)+\pi_1(m)=\pi(m)$.

Notice that given a $P$-position of the form $\{x,y,z\}$ with $x>y>z>m$, by definition, $(x,y)$, $(x,z)$, and $(y,z)$ are in $T_m$. Give an $P$-position of the form $\{x,x,y\}$ with $x$ and $y$ distinct and less than $m$, clearly  exactly one of $(x,y)$ and $(y,x)$ is in $m$. Also, all ordered pairs $(y,z)$ with $z<y<m$ and $x<m$ fall into one of these two cases. Therefore, $$\frac{m^2-m}{2}-h(m,m)=3\pi_3(m)+\pi_2(m).$$
Therefore, $$6\pi(m)=m^2-m-2h(m,m)+4\pi_2(m)+6\pi_1(m).$$
Therefore, 
$$\frac{\pi(n2^k)}{(n2^k)^2}=\frac{n^24^k-n2^k-2h({n2^k},n2^k)+4\pi_2(n2^k)+6\pi_1(n2^k)}{6(4^k)n^2},$$
so 
$$\frac{\pi(n2^k)}{(n2^k)^2}=\frac{1}{6}-\frac{1}{3n^2}\frac{h(n2^k,n2^k)}{4^k}-\frac{1}{6(2^k)n}+\frac{4\pi_2(n2^k)}{6(4^k)n^2}+\frac{\pi_1(n2^k)}{4^kn^2}.$$
Notice that $\frac{1}{6}$ is a constant, $-\frac{1}{3n^2}\frac{h({n2^k},n2^k)}{4^k}$ converges as $k$ goes to infinity by Lemmas \ref{limith} and \ref{degenerate}, and $-\frac{1}{6(2^k)n}$ converges to 0 as $k$ goes to infinity. For any $x$, there is only one $P$-position of the form $\{x,x,y\}$, and at most one $P$-position of the form $\{x,x,x\}$. Therefore, $\pi_2(n2^k)$ and $\pi_1(n2^k)$ are both less than or equal to $n2^k$, so $\frac{4\pi_2(n2^k)}{6(4^k)n^2}$ and $\frac{\pi_1(n2^k)}{4^kn^2}$ both converge to 0 as $k$ goes to infinity. Therefore, $\frac{\pi(n2^k)}{(n2^k)^2}$ converges as $k$ goes to infinity.

Further, we know that $\frac{\pi(n2^k)}{(n2^k)^2}$ does not converge to 0, since $\pi(n2^k)\geq \pi_3(n2^k)$ which is equal to the number of $P$-positions of the form $\{x,y,z\}$ with $x,y,z<m$, which is one sixth the number of ordered triples $(x,y,z)$ with $x,y,z<m$ such that $\{x,y,z\}$ is a $P$-position. For each pair $y,z<\frac{n2^{k}-|F|}{2}$, there exists an ordered triple $(x,y,z)$ with $x\leq y+z+|F|<n2^k$ and $y,z<n2^k$, such that $\{x,y,z\}$ is a $P$-position. There are $\frac{(n2^{k}-|F|)^2}{4}$ such pairs, so $$\pi(n2^k)\geq\pi_3(n2^k)\geq\frac{(n2^{k}-|F|)^2}{24}.$$
Therefore, 
$\frac{\pi(n2^k)}{(n2^k)^2}\geq\frac{(n2^{k}-|F|)^2}{24(n2^k)^2}$, which converges to $\frac{1}{24}$ as $k$ goes to infinity.

Therefore, for any positive integer $n$, $\displaystyle\lim_{k\rightarrow\infty}\frac{\pi(n2^k)}{(n2^k)^2}$ converges to a nonzero constant.
\end{proof}

\section{Concluding remarks}
We have thus shown here that the class of combinatorial games CIS-Nim obeys a form of scale invariance (period-two scale invariance).  The existence of such scaling properties in combinatorial games had been previously hinted at using renormalization techniques adapted from physics.  However, such techniques were nonrigorous in nature; the present work is the first  formal characterization of scaling in this context.  Additionally, it has been demonstrated that certain properties of combinatorial games persist under perturbations (the perturbations here being defined by the forbidden set  $F$), and hence are `generic' in the sense of dynamical systems theory.

That said, the version of the period-two scale invariance proven in this paper was not the strongest version possible. A much stronger version, which is also appears to be true, would allow more general regions than the set of points $\{x,y,z\}$ with $\{x,y,z\}<n.$ We therefore conjecture a stronger version of the period-two scale invariance: 
\begin{con}\label{genregion}
Given any instance $\mathit{Nim}-F$ of CIS-Nim and any open set $S\subseteq\mathbb{R}^3,$ let $\pi(R,k)$ be the number of $P$-positions of the form $\{x2^k,y2^k,z2^k\},$ with $x,y,z\in \mathbb{Q}$ and $(x,y,z)\in S$. Then $\displaystyle\lim_{k\rightarrow\infty}\frac{\pi(R,k)}{4^k}$ converges.
\end{con}
Alternatively, we could make this statement stronger by considering more general versions of the game of Nim. This can be done by considering the piles in Nim to be labeled, so the positions are ordered triples. This would allow for non-symmetric forbidden sets. We could also consider Nim played with a arbitrary number of piles. We conjecture that that this generalization will also preserve the period-two scale invariance.
\begin{con}\label{genheaps}
Let $m$ be any positive integer, let $F$ be a set of positions in $m$-Heap Nim with labeled piles. Given any open set $S\subseteq\mathbb{R}^m,$ let $\pi(R,k)$ be the number of $P$-positions of $m$-Heap $\mathit{Nim}-F$ of the form $2^kv$ with $v\in S$. Then $\displaystyle\lim_{k\rightarrow\infty}\frac{\pi(R,k)}{2^{(n-1)k}}$ converges.
\end{con}
Finally, this general notion of a Cofinite Induced Subgraph Games can be applied to any other impartial combinatorial games. Since other games do not necessarily satisfy a period-two scale invariance, this result will not generalize to the CIS version of most of these other games. However, by analyzing the cofinite induced subgraphs of a game graph, we learn which properties of the structure of the $P$-positions are unstable and dependent on a finite set of end game positions, and which properties stable and inevitable regardless of the details of the end game. 
\nocite{*}

Acknowledgments:  ASL's
research has been supported in part through a W.M. Keck
Foundation research grant. EJF's research has been supported in part by the NSF under grant CDI-0835706.

\end{document}